\documentclass[11pt, a4paper]{amsart}

\usepackage{lineno,hyperref}
\usepackage{geometry}
\usepackage{amssymb,amsmath}
\usepackage{amsfonts}
\usepackage{subcaption}
\usepackage[english]{babel}
\usepackage{amsmath,amssymb,eucal}
\usepackage{hyperref}
\usepackage{pgfplots}
\usepackage{caption}
\usepackage{tikz}
\usetikzlibrary{calc,backgrounds,positioning,decorations.pathmorphing,shapes.misc,matrix,through,shapes,patterns}%
\usepackage{float}

\newtheorem{definition}{Definition}

\newtheorem{theorem}{Theorem}

\newtheorem{proposition}{Proposition}
\newtheorem{lemma}{Lemma}%
\newtheorem{example}{Example}%

\newcommand{\dsum}{\displaystyle\sum}

\def\supp{\mathrm{supp}}
\def\Fc{\mathcal{F}}
\def\R{\mathbb{R}}
\def\N{\mathbb{N}}
\def\Kc{\mathcal{K}}

\let\origmaketitle\maketitle
\def\maketitle{
  \begingroup
  \def\uppercasenonmath##1{} 
  \let\MakeUppercase\relax 
  \origmaketitle
  \endgroup
}
\begin{document}

		\title[A fully fuzzy framework for minimax MILP]{\large On a fully fuzzy framework for minimax mixed integer linear programming}

		\author[Manuel Arana-Jim\'enez \MakeLowercase{and} V\'ictor Blanco]{{\large Manuel Arana-Jim\'enez$^\dagger$ and V\'ictor Blanco$^\ddagger$}\medskip\\
$^\dagger$Dpt.\ Statistics \& Operational Research, Universidad de C\'adiz\\
$^\ddagger$Dpt. Quant. Methods for Economics \& Business, Universidad de Granada
}

\address{M. Arana-Jim\'enez: Dpt.\ Statistics \& Operational Research, Universidad de C\'adiz}
\email{manuel.arana@uca.es}
\address{V. Blanco: Dpt. Quant. Methods for Economics \& Business, Universidad de Granada}
\email{vblanco@ugr.es}


		\begin{abstract}
			In this work, we present a modeling framework for minimax mixed 0-1 fuzzy linear problems. It is based on extending the usual rewriting of crisp minimax problems via auxiliary variables to model the maximum of a finite set of fuzzy linear functions. We establish that the considered problem can be equivalently formulated as a multiple objective mixed integer programming problem. The framework is applied to a fully fuzzy version of the capacitated center facility location problem.	\end{abstract}
		
		\keywords{Minimax optimization, Fully fuzzy linear programming problem, fuzzy numbers, multiobjective optimization, Center Facility Location.}
\maketitle 

\section{Introduction}

Minimax optimization is a widely studied modeling strategy in Decision Theory that has been applied in many different fields. Minimax decision strategies are desirable when the decision maker is risk-averse or when the average cost is less important than ensuring smaller maximum costs. One of the most popular results involving minimax decisions is the one due to Von Neumann \cite{vonneuman} in which is stated that one can find Nash Equilibria in Zero-Sum games using minimax optimization. That result is considered as the cornerstone of the developments performed in Non Cooperative Game Theory in the last decades. Several applications of minimax optimization can be found in the literature: facility location \cite{garfinkel,regret}, transportation \cite{ahuja86}, flow \cite{CCTT09}, scheduling \cite{graham}, resource allocation \cite{Kaplan74}, portfolio selection \cite{young} only to mention a few. In general, minimax optimization problems, although its objective function is not linear, but piecewise linear, it can be rewritten as a problem with the similar shape of the original one but when instead of minimizing the maximum operator of a finite set of functions, one only need to minimize a linear objective function after adding to the problem some extra constraints and a new auxiliary variable. Hence, the complexity of the problem is inherited to minmax problems. In particular, minimax mixed integer linear programming is NP-hard as its minisum version, but polynomially solvable when the number of integer variables is fixed \cite{lenstra}.

In this paper, we analyze minimax mixed integer linear programming problems in which some kind of imprecision is assumed to the parameters of the models and also to the decision variables. Such an imprecision is modeled by considering that the parameters and variables are fuzzy numbers. Decision making under a fuzzy environment was introduced by Bellman and Zadeh \cite{bellman} in the seventies. Such an approach is well-known and has been adopted by researchers in fuzzy optimization problems (see e.g., \cite{campos89,fmoknp,ftrans,kasperski07,fspp,ganesan06,maleki00,maleki02,ebra10}). However, it is usual that not all the elements of a  fuzzy linear problem are fuzzy sets (see \cite{herrera93,herrera-verdegay95,herrera-verdegay91,herrera-verdegay96}).  In \cite{lofti09}, the authors point out that there is no general method for finding the fuzzy optimal solution of fully fuzzy linear programming (FFLP) problems, and proposed a new method for solving FFLP problems with equality constraints when the parameters and variables are assumed to be symmetric fuzzy numbers. However, in \cite{kumar11}  the method proposed in \cite{lofti09} is criticized because its complexity and the authors claimed out that the obtained solutions are approximated, but not exact. A new method for finding the fuzzy optimal solution of (FFLP) problems with equality constraints, with triangular fuzzy numbers involved, although they use ranking function (see \cite{arana15}, and the references there in) to compare the objective function values. In this way, Khan et al. \cite{khan13} deal with (FFLP) with inequalities, and they also compare the objective function values via ranking functions (see also \cite{bhardwaj14,khan17}). Das and G\"o\c cken \cite{das2014} apply a ranking method for the reviewer assignment problem. Concerning the complexity, FFLP is NP-hard, since its crisp version is. However, in \cite{BP15} it is proved that if the fuzzy numbers involved in the problem are totally described by a finite ranking system, certain encoding of the solutions can be found in polynomial-time when the number of integer variables is fixed, and the solutions can be enumerated using a polynomial delay algorithm.

Here, we analyze fully fuzzy versions of minimax mixed integer linear programming problems. As mentioned above, the crisp version of the problem can be easily formulated as another mixed integer programming by identifying the maximum of a finite set of objective function with a new auxiliary variable and adequately representing it in the set of constraints. Nevertheless, in a fuzzy environment, each objective to be considered in the minimax approach is itself a fuzzy number. Hence, one needs a suitable representation of the maximum of a finite set of fuzzy sets to replicate the crisp strategy. We introduce the notion of minimal upper bound of a set of fuzzy numbers, and prove some structural properties about it. This notion allows us to extend the classical minimax optimization problems to their fuzzy counterpart.

The paper is organized as follows. In Section \ref{sec:1} some notation and preliminary results are stated. Section \ref{sec:2} is devoted to the analysis of fully fuzzy minimax mixed integer programming problems. There, we present two different fuzzy optimization models for the problem and we prove the equivalence between them. In Section \ref{sec:3}, the problem is equivalently reformulated as a three-objective mixed integer programming problem. In Section \ref{sec:4} we apply the framework to a well known problem in Location Analysis, the center capacitated facility location problem. Finally, in Section \ref{sec:5} we draw some conclusions of the paper.


\section{Preliminaries and Notation}
\label{sec:1}
In this section we introduce the notation used through the rest of the paper. We also recall here some preliminary results on fuzzy sets that will be useful in our development.

A fuzzy set defined on $\mathbb{R}^{n}$ is a mapping (also known as \textit{the membership function})
$\mu: \mathbb{R}^{n}\rightarrow \lbrack 0,1]$, which for any $x\in \R^n$, represent the degree of truth of being such a value. For each fuzzy set
$\mu$, we denote its $\alpha $-level set
by $[\mu]^{\alpha }=\{x\in \mathbb{R}^{n}$ $|$ $\mu(x)\geq \alpha \}$ for any $%
\alpha \in [0,1]$. The support of $\mu$ is defined as $\supp(u)=\{x\in \mathbb{R}^{n}$ $|$ $\mu(x)>0\}$. Observe that the closure of
$\supp(\mu)$ defines the $0$-level set of $\mu$, .i.e. $[\mu]^{0}=cl(\supp(\mu))$
where here $cl(M)$ stands for the closure of the subset $M\subset \mathbb{R}^{n}$. Fuzzy numbers are particular fuzzy sets defined over $\R$ (see \cite{Dubois78,Dubois80}).

\begin{definition}
	A fuzzy set $\mu$ on $\mathbb{R}$ is said to be a fuzzy number if the following properties are verified:
	\begin{enumerate}
		\item There exists $x_0 \in \mathbb{R}$ such that $\mu(x_0)=1$,
		\item $\mu$ is an upper semi-continuous function,
		\item $\mu(\lambda x+(1-\lambda )y)\geq \min\{\mu(x),\mu(y)\},$ for all $x,y\in \mathbb{R}$ and  $\lambda \in \lbrack 0,1]$, and
		\item $[\mu]^0$ is compact.
	\end{enumerate}
\end{definition}

Let $\mathcal{F}$ denote the family of all fuzzy numbers and let $\mathcal{K}$ denote the family of all bounded closed
intervals in $\mathbb{R}$, i.e.,
\begin{equation*}
\mathcal{K}=\left\{ \left[ \underline{a},\overline{a}\right] \;|\;%
\underline{a},\overline{a}\in \mathbb{R}\mbox{ and
}\underline{a}\leq \overline{a}\right\} ,
\end{equation*}
It is clear that for any $\mu\in \mathcal{F}$, its level sets $[\mu]^{\alpha }\in
\mathcal{K}$ for all $\alpha \in [0,1]$. Then,  $[\mu]^{\alpha }=\left[ \underline{\mu}_{\alpha }, \overline{\mu}_{\alpha }\right] ,$ $\underline{\mu}_{\alpha },\overline{\mu}%
_{\alpha }\in \mathbb{R}$ for all $\alpha \in [0,1]$.

There are many parametrical families of fuzzy numbers that have been applied to measure imprecision in several situations. Among the most popular one can find the L-R, trapezoidal, triangular, gaussian, quasi-gaussian, quasi-quadric, exponential, quasi-exponential, and singleton fuzzy numbers (see  \cite{hans05} for a complete description of these families). One of the most used families of fuzzy numbers, because of its easy modeling and interpretation, are triangular fuzzy numbers (see, for instance, \cite{Dubois78, Dubois80,kaufmann85,khan13,lofti09,stefanini06}). 
\begin{definition}\label{triangular number}
	A fuzzy number  $\mu$ is said a triangular fuzzy number if there exist $\mu^-$, $\hat{\mu}$, $\mu^+ \in \R$ with $\mu^-<\hat \mu < \mu^+$, such that:
	$$
	\mu(x)=\left\{
	\begin{array}{lr}
	\dfrac{x-\mu^-}{\hat \mu- \mu^-} &\mbox{ if } \mu^- \le x\le \hat \mu,\\
	\dfrac{\mu^+-x}{\mu^+-\hat \mu} &\mbox{ if } \hat \mu< x\le \mu^+,\\
	0 & \mbox{otherwise}.
	\end{array}
	\right.$$
	In such a case $\mu$ is denoted as the triplet $(\mu^-,\hat{\mu},\mu^+)$. 	\end{definition}
The set of triangular fuzzy numbers will be denoted as $\mathcal{T}$. Abusing of notation, the degenerated case in which $\mu^-=\hat \mu=\mu^+ \in \R$, i.e.:
	$$(\hat \mu, \hat \mu, \hat \mu) \equiv \mu(x)=\left\{
	\begin{array}{lr}
	1 &\mbox{ if } x=\hat{\mu},\\
	0 & \mbox{otherwise}.
	\end{array}
	\right.$$ is also considered a triangular fuzzy number which is identified with the crisp number $\hat{\mu}$. In particular, $\tilde{0}=(0,0,0)$ will allow to model the nonnegativity of a triangular fuzzy number.

In Figure \ref{fig:t0} we show the shape of a triangular fuzzy number.

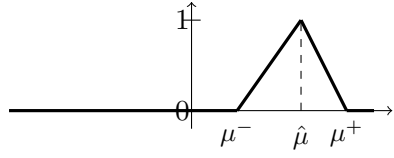
\begin{figure}[h]
\begin{center}
\begin{tikzpicture}[scale=1.2]
\draw[->] (-2,0)--(2.2,0);
\draw[->] (0,-0.2)--(0,1.2);

\node at (-0.1,1) {$1$};
\node at (-0.1,0) {$0$};

\draw (-0.1,1)--(0.1,1);
\draw (-0.1,0)--(0.1,0);

\draw[dashed] (1.2,0)--(1.2,1);

\draw[very thick] (-2,0)--(0.5,0);
\draw[very thick] (0.5,0)--(1.2,1);
\draw[very thick] (1.2,1)--(1.7,0);
\draw[very thick] (1.7,0)--(2,0);

\node at (0.5,-0.25) {\small$\mu^-$};
\node at (1.2,-0.29) {\small$\hat \mu$};
\node at (1.7,-0.25) {\small$\mu^+$};
\end{tikzpicture}
\end{center}
\caption{Graph of the triangular fuzzy number $(\mu^-,\hat{\mu},\mu^+)$.\label{fig:t0}}
\end{figure}

Note that given a triangular fuzzy number $\mu \equiv (\mu^-,\hat{\mu},\mu^+)$, its $\alpha$-levels have the following shape:
$$
[\mu]^\alpha=[\mu^-+(\hat\mu-\mu^-)\alpha, \mu^+-(\mu^+-\hat \mu)\alpha], \mbox{for all $\alpha\in[0,1]$. }
$$
Thus, it is not difficult to see that this shape of $\alpha$-levels sets completely characterizes triangular fuzzy numbers (see \cite{Goestschel}).

Given a fuzzy number $\mu=(\underline{\mu},\overline{\mu})$,  we say that $\mu$ is a nonnegative fuzzy number if  $\underline{\mu}_{0}\geq 0$. Hence, a nonnegative triangular fuzzy number $(\mu^-,\hat{\mu},\mu^+)$ is characterized by $\mu^-\geq 0$. The set of nonnegative fuzzy numbers is denoted by $\mathcal{T}_{\geq 0}$.

Interval arithmetic extend to fuzzy numbers. Let $\mu, \nu \in \mathcal{F}$ represented by their level sets$\left[ \underline{\mu}_{\alpha },\overline{\mu}_{\alpha }%
\right] $ and $\left[ \underline{\nu}_{\alpha },\overline{\nu}_{\alpha }\right] $%
, respectively. Let $\lambda \in \mathbb{R}$, then the addition $%
\mu+\nu$ and scalar multiplication $\lambda \nu$ are defined as follows:
$$
(\mu+\nu)(x)=\sup_{y+z=x}\min \{\mu(y),\nu(z)\}, \quad 
(\lambda \mu)(x)=\left\{
\begin{array}{ll}
\mu\left( \frac{x}{\lambda }\right) , & \hbox{if}\text{ \ }\lambda \neq 0, \\
0, & \hbox{if }\text{ \ }\lambda =0.%
\end{array}%
\right.
$$
Which, in terms of the $\alpha$-level sets it is equivalent to ,
\begin{equation}
\lbrack \mu+\nu]^{\alpha }=\left[ (\underline{\mu+\nu})_{\alpha },(\overline{\mu+\nu}%
)_{\alpha }\right] =\left[ \underline{\mu}_{\alpha
}+\underline{\nu}_{\alpha }\;,\;\overline{\mu}_{\alpha
}+\overline{\nu}_{\alpha }\right]   \label{Eqsum}
\end{equation}%
and
\begin{equation}
\lbrack \lambda \mu]^{\alpha }=\left[ (\underline{\lambda \mu})_{\alpha },(%
\overline{\lambda \mu})_{\alpha }\right] =\left[ \min \{\lambda \underline{\mu}%
_{\alpha },\lambda \overline{\mu}_{\alpha }\},\max \{\lambda \underline{\mu}%
_{\alpha },\lambda \overline{\mu}_{\alpha }\}\right] .  \label{Eqmul}
\end{equation}
for all $\alpha \in \lbrack 0,1]$.

The above operations on general fuzzy numbers are particularized to triangular fuzzy number in the following lemma whose proof is straightforward.

\begin{lemma}\label{lemma:1}
Let $\tilde{\mu}=(\mu^-,\hat{\mu},\mu^+)$ and $\tilde{\nu}=(\nu^-,\hat{\nu},\nu^+)$ two triangular fuzzy numbers and $\lambda \in \R$. Then:
\begin{itemize}
	\item[(i)]  $\tilde{\mu}+\tilde{\nu}=(\mu^-+\nu^-,\hat{\mu}+\hat{\nu},\mu^++\nu^+)$.
	\item[(ii)]  $\lambda \tilde{\mu}=\left\{\begin{array}{cl}
	(\lambda \mu^-,\lambda \hat{\mu},\lambda \mu^+) & \mbox{if $\lambda \ge 0$,}\\
	(\lambda \mu^+,\lambda \hat{\mu},\lambda \mu^-) & \mbox{if $\lambda < 0$.}
	\end{array}\right.$.
	\item[(iii)] If $\tilde{\nu}$ is a nonnegative triangular fuzzy number, then
	\begin{equation}
	\label{multiplication}
	\tilde{\mu}\tilde{\nu}=\left\{\begin{array}{ll}
	(\mu^-\nu^-,\hat{\mu}\hat{\nu},\mu^+\nu^+) & \mbox{ if } \mu^-\ge 0,\\
	(\mu^-\nu^+,\hat{\mu}\hat{\nu},\mu^+\nu^+) & \mbox{ if } \mu^-<0,\, \mu^+\ge 0,\\
	(\mu^-\nu^+,\hat{\mu}\hat{\nu},\mu^+\nu^-) & \mbox{ if } \mu^+<0.
	\end{array}\right.
	\end{equation}
\end{itemize}
\end{lemma}

In Figure \ref{fig:t1} we illustrate the results of the above operations with triangular fuzzy numbers.

\begin{figure}[h]
\begin{center}
\begin{tikzpicture}[scale=1]
\draw[->] (-1,0)--(4.2,0);
\draw[->] (0,-0.2)--(0,1.2);

\node at (-0.1,1) {$1$};
\node at (-0.1,0) {$0$};

\draw (-0.1,1)--(0.1,1);
\draw (-0.1,0)--(0.1,0);

\draw[very thick, dashed] (-1,0)--(0.5,0);
\draw[very thick, dashed] (0.5,0)--(1.2,1);
\draw[very thick, dashed] (1.2,1)--(1.7,0);
\draw[very thick, dashed] (1.7,0)--(4,0);

\draw[very thick, dotted] (-1,0)--(0.9,0);
\draw[very thick, dotted] (0.9,0)--(1.8,1);
\draw[very thick, dotted] (1.8,1)--(2,0);
\draw[very thick, dotted] (2,0)--(4,0);

\draw[very thick] (-1,0)--(1.4,0);
\draw[very thick] (1.4,0)--(3,1);
\draw[very thick] (3,1)--(3.7,0);
\draw[very thick] (3.7,0)--(4,0);

\node[above] at (1.2,1) {$\small \mu$};
\node[above] at (1.8,1) {$\small \nu$};
\node[above] at (3,1) {$\small \mu+\nu$};

\end{tikzpicture}~\begin{tikzpicture}[scale=1]
\draw[->] (-1.2,0)--(2.2,0);
\draw[->] (0,-0.2)--(0,1.2);

\node at (-0.1,1) {$1$};
\node at (-0.1,0) {$0$};

\draw (-0.1,1)--(0.1,1);
\draw (-0.1,0)--(0.1,0);

\draw[very thick, dashed] (-1,0)--(0.5,0);
\draw[very thick, dashed] (0.5,0)--(1.2,1);
\draw[very thick, dashed] (1.2,1)--(1.7,0);
\draw[very thick, dashed] (1.7,0)--(2,0);

\draw[very thick] (-1,0)--(0.25,0);
\draw[very thick] (0.25,0)--(0.6,1);
\draw[very thick] (0.6,1)--(0.85,0);
\draw[very thick] (0.85,0)--(2,0);

\draw[very thick] (-1,0)--(-0.85,0);
\draw[very thick] (-0.85,0)--(-0.6,1);
\draw[very thick] (-0.6,1)--(-0.25,0);
\draw[very thick] (-0.25,0)--(2,0);

\node[above] at (1.2,1) {$\small \mu$};
\node[above] at (0.6,1) {$\small \frac{1}{2}\mu$};
\node[above] at (-0.6,1) {$\small -\frac{1}{2}\mu$};

\end{tikzpicture}~\begin{tikzpicture}[scale=1]
\draw[->] (-1,0)--(4.2,0);
\draw[->] (0,-0.2)--(0,1.2);

\node at (-0.1,1) {$1$};
\node at (-0.1,0) {$0$};

\draw (-0.1,1)--(0.1,1);
\draw (-0.1,0)--(0.1,0);

\draw[very thick, dashed] (-1,0)--(0.5,0);
\draw[very thick, dashed] (0.5,0)--(1.2,1);
\draw[very thick, dashed] (1.2,1)--(1.7,0);
\draw[very thick, dashed] (1.7,0)--(4,0);

\draw[very thick, dotted] (-1,0)--(0.9,0);
\draw[very thick, dotted] (0.9,0)--(1.8,1);
\draw[very thick, dotted] (1.8,1)--(2,0);
\draw[very thick, dotted] (2,0)--(4,0);

\draw[very thick] (-1,0)--(0.45,0);
\draw[very thick] (0.45,0)--(2.16,1);
\draw[very thick] (2.16,1)--(3.4,0);
\draw[very thick] (3.4,0)--(4,0);

\node[above] at (1.2,1) {$\small \mu$};
\node[above] at (1.8,1) {$\small \nu$};
\node[above] at (3,1) {$\small \mu\nu$};

\end{tikzpicture}
\end{center}
\caption{Operations of triangular fuzzy numbers.\label{fig:t1}}
\end{figure}
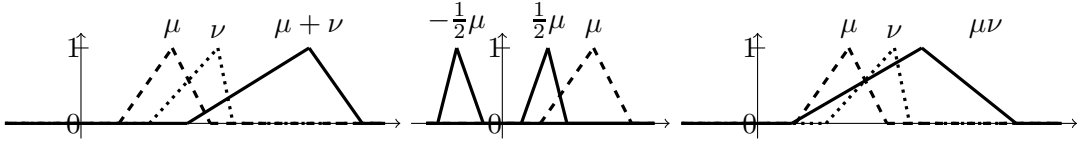

In order to compare two fuzzy numbers, there exist several definitions, extending binary relations on intervals (see \cite{guerra12}). In this regard, given $\mu,\nu \in\Fc$, we write their $\alpha$-levels as
$\mu_{\alpha}=[\underline{\mu}_{\alpha},\overline{\mu}_{\alpha}]\in \Kc$ and
$\nu_{\alpha}=[\underline{\nu}_{\alpha},\overline{\nu}_{\alpha}]\in \Kc$, respectively, for all $\alpha\in [0,1]$.
\begin{definition}\label{def partial orders}
	Given $u, v\in\Fc$, we say that:
	\begin{itemize}
		\item[(i)] $\mu\lessapprox \nu$ if and only if 
		$\underline{\mu}_{\alpha}\leq \underline{\nu}_{\alpha}$ and $\overline{\mu}_{\alpha}\leq \overline{\nu}_{\alpha}$, for all $\alpha\in [0,1]$,
		\item[(ii)] $\mu\prec \nu$ if and only if 	$\underline{\mu}_{\alpha}< \underline{\nu}_{\alpha}$ and  $\overline{\mu}_{\alpha}< \overline{\nu}_{\alpha}$, for all $\alpha\in [0,1]$,
		\item[(iii)] $\mu\preceq \nu$ if and only if $\mu\lessapprox \nu$ and $\mu\ne \nu$.
		
	\end{itemize}
\end{definition}

Checking whether any of the above binary relations is verified for a pair of triangular fuzzy numbers, can be done using only the extreme points of the two fuzzy numbers, as stated in the following result.
\begin{theorem}\label{order characterization}
	Given two triangular fuzzy numbers $\tilde{\mu}=(\mu^-, \hat{\mu}, \mu^+)$ and $\tilde{\nu}=(\nu^-, \hat{\nu}, \nu^+)$, it follows that 
	\begin{itemize}
		\item [(i)] $\tilde{\mu}\prec\tilde{\nu}$ if and only if $\mu^-<\nu^-$, $\hat{\mu}<\hat{\nu}$ and $\mu^+<\nu^+$.
		\item[(ii)] $\tilde{\mu}\lessapprox\tilde{\nu}$ if and only if $\mu^-\le \nu ^-$, $\hat{\mu}\le \hat{v}$ and $\mu^+\le v^+$.
		\item[(iii)] $\tilde{\mu}\preceq\tilde{\nu}$ if and only if $\mu^-\le \nu^-$, $\hat{\mu}\le \hat{\nu}$ and $\mu^+\le \nu^+$, where some inequality is strict.
	\end{itemize}
\end{theorem}
\begin{proof}
	(i) First, we suppose that $\tilde{\mu}\prec\tilde{\nu}$. The $\alpha$-level sets are:
	 $$
	 [\tilde{\mu}]^\alpha=[\mu^-+(\hat{\mu}-\mu^-)\alpha, \mu^+-(\mu^+-\hat{\mu})\alpha] \text{ and } [\tilde{\nu}]^\alpha=[ \nu^-+(\hat{\nu}-\nu^-)\alpha,  \nu^+-(\nu^+-\hat{\nu})\alpha]
	 $$
	 for all $\alpha\in [0,1]$.
	 Then, by Definition \ref{def partial orders}, we get that $\underline{\tilde{\mu}}_{\alpha}< \underline{\tilde{\nu}}_{\alpha}$ and  $\overline{\tilde{\mu}}_{\alpha}< \overline{\tilde{\nu}}_{\alpha}$, for all $\alpha\in [0,1]$, i.e., 
	$$
	\mu^-+(\hat{\mu}-\mu^-)\alpha <\nu^-+(\hat{\nu}- \nu^-)\alpha  \text{ and }  \mu^+-(\mu^+-\hat{\mu})\alpha <\nu^+-(\nu^+-\hat{\nu})\alpha,
	$$
	for all $\alpha\in [0,1]$. Hence, for $\alpha=0$, we get that $\mu^-<\nu^-$ and $\mu^+<\nu^+$, and for $\alpha=1$,  that $\hat{\mu}<\hat{\nu}$.
	
	Conversely, let us suppose that $\mu^-<\nu^-$, $\hat{\mu}<\hat{\nu}$ and $\mu^+<v^+$. Let us check whether the conditions $\underline{\tilde{\mu}}_{\alpha}< \underline{\tilde{\nu}}_{\alpha}$ and  $\overline{\tilde{\mu}}_{\alpha}< \overline{\tilde{\nu}}_{\alpha}$ are verified for any $\alpha\in [0,1]$. Let $\alpha\in [0,1]$. By hypothesis we have that:
	$$
	(1-\alpha)\mu^- <(1-\alpha)v^-\quad\mbox{and}\quad \hat{\mu}\alpha <\hat{\nu}\alpha.
	$$
	Combining these two inequalities, we have that
	$$
	(1-\alpha)\mu^- +\hat{\mu}\alpha<(1-\alpha)\nu^- +\hat{\nu}\alpha.
	$$
	On the other hand, we also have that $(1-\alpha)\mu^+ <(1-\alpha)\nu^+$ and $\hat{\mu}\alpha <\hat{\nu}\alpha,$ implying that that $(1-\alpha)\mu^+ +\hat{\mu}\alpha<(1-\alpha)\nu^+ +\hat{\nu}\alpha,$. Therefore,  $\overline{\tilde{\mu}}_{\alpha}< \overline{\tilde{\nu}}_{\alpha}$. Thus, $\tilde{\mu}\prec\tilde{\nu}$. The proofs of (ii) and (iii) is similar to (i) and are left for the reader.
\end{proof}

Finally, we are interested in suitable representations of the maximum of a finite set of triangular fuzzy numbers. Several definitions have been proposed in the literature, the most popular being the maximum approximation concept \cite{palacios17,fortemps97}, in which the maximum of the finite set of triangular fuzzy numbers $\{\tilde{a}_1,\tilde{a}_2,\dots,\tilde{a}_r\}$ is given by the following triangular fuzzy number:
\begin{equation}\label{max}
\max \{\tilde{a}_1,\tilde{a}_2,\dots,\tilde{a}_r\}=(\max \{a^{-}_{1}, a^{-}_{2},\dots ,a^{-}_{r}\},\max \{\hat{a}_1,\hat{a}_2,\dots ,\hat{a}_r\}, \max \{a^{+}_{1}, a^{+}_{2},\dots ,a^{+}_{r}\}).
\end{equation}
One of the main critics to this definition is that the maximum of a finite set of triangular fuzzy numbers is not always an element of the set of triangular fuzzy numbers. However, we have the following relationship whose proof is straightforward from Theorem \ref{order characterization} and (\ref{max}).

\begin{proposition}
	Given the fuzzy triangular numbers $\tilde{a}_1,\tilde{a}_2,\dots,\tilde{a}_r$, with $r\in\N$, if there exist $i_0\in\{1,\dots, r\}$ such that $\tilde{a}_i\lessapprox\tilde{a}_{i_0}$, for all $i\in \{1,\dots, r\}$, then
	$$\text{max } \{\tilde{a}_1,\tilde{a}_2,\dots,\tilde{a}_r\}= \tilde{a}_{i_0}.$$
\end{proposition}

\begin{definition}
	We say that a fuzzy number $\tilde{\mu}$ is an upper bound for a finite set of fuzzy numbers, $S$, if $\tilde{s}\lessapprox \tilde{\mu}$ for all $\tilde{s}\in S$. The set of upper bounds for $S$ will be denoted by $\mathcal{U}(S)$
\end{definition}
In case that $S$ is a finite set of triangular fuzzy numbers, the following result states that $\mathcal{U}(S)\neq \emptyset$.

\begin{proposition}\label{prop:2}
	Let $S=\{\tilde{a}_1,\tilde{a}_2,\dots,\tilde{a}_r\}$ a finite set of triangular fuzzy numbers and let $\Theta(S) = \max \{\tilde{a}_1,\tilde{a}_2,\dots,\tilde{a}_r\}$. Then:
	\begin{enumerate}
	\item $\Theta(S) \in \mathcal{U}(S)$, and
	\item  $\Theta(S)$ is the unique $\lessapprox$-minimal element of $\mathcal{U}(S)$, that is, there not exists $\tilde{\mu} \in \mathcal{U}(S)\backslash\{\Theta(S)\}$ such that $\tilde{\mu} \lessapprox \Theta(S)$. 
	\end{enumerate}
\end{proposition}
\begin{proof}
Theorem \ref{order characterization} and \eqref{max} imply that  $\tilde{s}\lessapprox \Theta(S)$ for all $\tilde{s}\in S$. Hence, $\Theta(S)=(\Theta(S)^-, \widehat{\Theta(S)}, \Theta(S)^+) \in \mathcal{U}(S)$. Assume now that there exists $\tilde{\mu}$ an upper bound for the set $S$ such that $\tilde{\mu}\preceq \Theta(S)$, that is, $\tilde{\mu}\lessapprox \Theta(S)$ and either ${\mu}^-<\Theta(S)^-$ or  $\hat{\mu}^{-}<\widehat{\Theta(S)}$ or  ${{\mu}}^+< \Theta(S)^+$. Without loss of generality, let us suppose that  ${{\mu}}^-< \Theta(S)^-=\max\{a^{-}_{1}, a^{-}_{2},\dots ,a^{-}_{r}\}$ (the remainder cases can be similarly analyzed). Since $\tilde{\mu}$ is an upper bound of $S$, we get that $\max\{a^{-}_{1}, a^{-}_{2},\dots ,a^{-}_{r}\}\le {\mu}^-$, contradicting the definition of $\tilde{\mu}$. Thus, $\Theta(S)$ a minimal upper bound for the set $S$. 

Let us assume now that there exists $\tilde{\nu}=(\nu^-,\hat{\nu}, \nu^+) \in \mathcal{U}(S)\backslash\{\Theta(S)\}$ such that there not exists $\tilde{\mu} \in \mathcal{U}(S)\backslash\{\tilde{\nu}\}$ such that $\tilde{\mu} \lessapprox \tilde{\nu}$. In particular, take $\tilde{\mu}=\Theta(S)$. Since $\tilde{\nu}\in \mathcal{U}(S)$, by Theorem \ref{order characterization} it follows that
$$\Big(\max\{a^{-}_{1}, a^{-}_{2},\dots ,a^{-}_{r}\},\max\{\hat{a}_1,\hat{a}_2,\dots ,\hat{a}_r\},\max\{a^{+}_{1}, a^{+}_{2},\dots ,a^{+}_{r}\}\Big)\lessapprox \tilde{\nu},$$

 what implies that $\Theta(S)\lessapprox \tilde{\nu}$. Since $\tilde{\nu}\ne\Theta(S)$, we get that $\Theta(S) \preceq \tilde{\nu}$, so $\tilde{\nu}$ is not a minimal upper bound for $S$.
\end{proof}

From the above result, given a finite set of triangular fuzzy numbers, $S=\{\tilde{a}_1,\tilde{a}_2,\dots,\tilde{a}_r\}$, the triangular fuzzy number $\Theta(S)$ is called the \textit{minimal upper bound of $S$}, and is defined as the triplet:
\begin{equation}\label{mub}
\Theta(S)=\Big(\max\{a^{-}_{1}, a^{-}_{2},\dots ,a^{-}_{r}\},\max\{\hat{a}_1,\hat{a}_2,\dots ,\hat{a}_r\},\max\{a^{+}_{1}, a^{+}_{2},\dots ,a^{+}_{r}\}\Big).
\end{equation}

Observe that in case $\tilde{a}_i$ are degenerated fuzzy numbers, i.e., $\tilde{a}_i\equiv (a_i,a_i,a_i)$ for all $i=1, \ldots, r$, ${\Theta}(\{\tilde{a}_1,\tilde{a}_2,\dots,\tilde{a}_r\})$ coincides with the crisp number $\max\{a_1, \ldots, a_r\}$. Thus, the minimal upper bound naturally extends the notion of maximum of a finite set of real numbers.

\section{Minimax fully fuzzy linear programming problem}
\label{sec:2}
In this section we present a fuzzy version of minimax mixed linear integer optimization problems, in which the notion of minimal upper bound, defined in the previous section, is involved. 

Recall that a minimax mixed linear integer programming problem consists of the following optimization problem:
\begin{align}
\min & \max \{c_1^tx, \ldots, c_r^tx\}\label{mm0}\\
\mbox{s.t. } & A x \;\; {\leq\choose =} \;\; b,\\
& x_i \geq 0, \forall i \in N,\\
& x_i \in \{0,1\}, \forall i \in Z.\label{mmf}
\end{align}
where $A \in \mathbb{R}^{m\times n}$ and $b\in \R^n$. The sets of indices $N, Z \subseteq \{1, \ldots, n\}$ indicate which variables are assumed to be nonnegative ($N$) or binary ($B$). Here, ${\leq\choose =}$ stands for the an element in $\{\leq, =\}^m$ indicating if the corresponding  constraint is in inequality or equation form. Several applications of minimax optimization can be found in the literature: facility location \cite{garfinkel,regret}, transportation \cite{ahuja86}, flow \cite{CCTT09}, scheduling \cite{graham}, resource allocation \cite{Kaplan74}, portfolio selection \cite{young} only to mention a few.

It is usual, to solve minimax problems as the above by rewriting the problem as an standard mixed integer linear programming problem by adding a new auxiliary variable $\theta$ represent $\max \{c_1^tx, \ldots, c_r^tx\}$. Then, \eqref{mm0}-\eqref{mmf} is equivalent to:
\begin{align}
\min & \;\; \theta\label{mm00}\\
\mbox{s.t. } & c_j^tx \leq \theta, \forall j=1, \ldots, r,\\
&A x \;\; {\leq\choose =} \;\; b,\\
& \alpha \in \R,\\
& x_i \geq 0, \forall i \in N,\\
& x_i \in \{0,1\}, \forall i \in Z.\label{mmff}
\end{align}

In what follows, we consider a fuzzy version of \eqref{mm0}--\eqref{mmf} in which the parameters in the objective functions and constraints, as well as the nonnegative variables involved in the problem are triangular fuzzy numbers. 

Let us now consider $\tilde{A}=(\tilde{a}_{ij}) \in \mathcal{T}^{m\times n}$, $\tilde{b}=(\tilde{b}_1, \ldots, \tilde{b}_m) \in \mathcal{T}^m$ and $\tilde{c}=(\tilde{c}_1, \ldots, \tilde{c}_r) \in \mathcal{T}^r$.  The minimax fully fuzzy linear programming (MFFLP) problem is formulated as follows:

\begin{align}
\min & \;\;  \Theta(\{\tilde{c}_1^t\tilde{x}, \ldots, \tilde{c}_r^t\tilde{x}\}) \label{fuzzy:0}\tag{MMFFLP}\\
\mbox{s.t. } & \tilde{A} \tilde{x} \;\; {\preceq\choose =} \;\; \tilde{b},\\
& \tilde{x}_i \in \mathcal{T}_{\geq 0}., \forall i \in N,\\
& \tilde{x}_i\equiv(x_i,x_i,x_i) \in \{0,1\}, \forall i \in Z.\label{fuzzy:f}
\end{align}
where $\Theta(\{\tilde{c}_1^tx, \ldots, \tilde{c}_r^tx\})$ is the minimal upper bound operator defined in \eqref{mub}.

Note that apart from the \textit{fuzzyfication} of the parameters in \eqref{mm0}--\eqref{mmf}, the crisp and the fuzzy minimax problems also differ in the binary relation in the constraints and the objective function. In case of the constraints the $\leq$ relation is substituted by the binary relation $\preceq$ between fuzzy while in the objective function, the maximum operator turns into the minimal upper bound operator of fuzzy numbers.

Based on the relationships introduced in Definition \ref{def partial orders}, we propose the following concept of nondominated solution for the considered problem (MMFFLP).

\begin{definition}\label{Nondominated solution}
	Let $\tilde{x}^{*}$ be a feasible solution for \eqref{fuzzy:0}. We say that $\tilde{x}^{*}$ is \textit{fuzzy optimal solution} of \eqref{fuzzy:0} if there not exist a feasible solution $\tilde{x}$ of \eqref{fuzzy:0} such that $\Theta(\{\tilde{c}_1^t\tilde{x}, \ldots, \tilde{c}_r^t\tilde{x}\}) \preceq \Theta(\{\tilde{c}_1^t\tilde{x}^*, \ldots, \tilde{c}_r^t\tilde{x}^*\})$.
\end{definition}
In order to get a more simple formulation of \eqref{fuzzy:0}, and in analogy to the crisp problem, we include a new fuzzy variable $\tilde{\theta}=(\theta^-,\hat{\theta},\theta^+)$:

\begin{align}
\min & \;\;  \tilde{\theta} \label{fuzzy1:0}\tag{MMFFLP1}\\
\mbox{s.t. } &\tilde{c}_j^t\tilde{x} \lessapprox \tilde{\theta}, \forall j=1, \ldots, r, \label{2:0}\\
& \tilde{A} \tilde{x} \;\; {\preceq\choose =} \;\; \tilde{b},\label{2:1}\\
& \tilde{\theta} \in \mathcal{T}, \label{2:2}\\
& \tilde{x}_i \in \mathcal{T}_{\geq 0}., \forall i \in N,\label{2:3}\\
& \tilde{x}_i\equiv(x_i,x_i,x_i) \in \{0,1\}, \forall i \in Z.\label{fuzzy1:f}
\end{align}

The following result states the relationwhip between the fuzzy optimal solutions of \eqref{fuzzy:0} and \eqref{fuzzy1:0}.
\begin{theorem}\label{Pareto between (MMFFLP) and (MMFFLP)}
Let $\tilde{x}^*$ be a feasible solution of \eqref{fuzzy:0} and $\tilde{\theta}^*=\Theta(\{\tilde{c}_1^t\tilde{x}^*, \ldots, \tilde{c}_r^t\tilde{x}^*\}) $. Then, $\tilde{x}^*$ is a fuzzy optimal solution of \eqref{fuzzy:0} if and only if $(\tilde{x}^*,\tilde{\theta}^*)$ is a fuzzy optimal solution of \eqref{fuzzy1:0}.
\end{theorem}
\begin{proof}
Let $\tilde{x}^{*}$ a fuzzy optimal solution of \eqref{fuzzy:0} and $\tilde{\theta}=\Theta(\{\tilde{c}_1^t\tilde{x}^*, \ldots, \tilde{c}_r^t\tilde{x}^*\}) $.  Observe that by definition $\tilde{c}_j^t\tilde{x}^*\lessapprox \tilde{\theta}^*$ for all $j=1, \ldots, r$.  Then, \eqref{2:0}--\eqref{fuzzy1:f} are verified, so $(\tilde{x}^*,\tilde{\theta}^*)$ is a feasible solution of \eqref{fuzzy1:0}. To show that $(\tilde{x}^*,\tilde{\theta}^*)$ is a fuzzy optimal solution, let $(\tilde{x},\tilde{\theta})$ a feasible solution for \eqref{fuzzy1:0}. In particular, $\tilde{x}$ is a feasible solution of \eqref{fuzzy:0}. Since $\tilde{x}^{*}$ is a fuzzy optimal solution of \eqref{fuzzy:0}, we have that
	\begin{equation}\label{mub 2}
	\Theta(\{\tilde{c}_1^t\tilde{x}, \ldots, \tilde{c}_r^t\tilde{x}\})\not\preceq\Theta(\{\tilde{c}_1^t\tilde{x}^*, \ldots, \tilde{c}_r^t\tilde{x}^*\})\lessapprox\tilde{\theta}^*.
	\end{equation}  

Since $(\tilde{x},\tilde{\theta})$ verifies \eqref{2:0}, we get that $\tilde{\theta} \in \mathcal{U}(\{\tilde{c}_1^t\tilde{x}, \ldots, \tilde{c}_r^t\tilde{x}\})$. Then, by Proposition \ref{prop:2} we have that 
$\Theta(\{\tilde{c}_1^t\tilde{x}, \ldots, \tilde{c}_r^t\tilde{x}\})\preceq\tilde{\theta}$. 
which combined with (\ref{mub 2}) implies that $\tilde{\theta}\not\preceq\tilde{\theta}^*$. Thus, $(\tilde{x}^*,\tilde{\theta}^*)$ is a fuzzy optimal solution of \eqref{fuzzy1:0}. 

Let us now suppose that $(\tilde{x}^*,\tilde{\theta}^*)$ is a fuzzy optimal solution of \eqref{fuzzy1:0}. Then it is not difficult to see that $\tilde{x}^{*}$ is feasible for \eqref{fuzzy1:0}. Let  $\tilde{x}$ be a feasible solution of \eqref{fuzzy:0} with $\tilde{\theta}=\Theta(\{\tilde{c}_1^t\tilde{x}, \ldots, \tilde{c}_r^t\tilde{x}\}){\preceq}\Theta(\{\tilde{c}_1^t\tilde{x}^*, \ldots, \tilde{c}_r^t\tilde{x}^*\})$.

Clearly, $(\tilde{x},\tilde{\theta})$ is a feasible solution of \eqref{fuzzy1:0}. Since $(\tilde{x},\tilde{\theta})$ verifies \eqref{2:0}, $\Theta(\{\tilde{c}_1^t\tilde{x}, \ldots, \tilde{c}_r^t\tilde{x}\}) \preceq \tilde{\theta}^*$, and then $\tilde{\theta} \lessapprox \tilde{\theta}^*$. Thus, $\tilde{\theta} \preceq \tilde{\theta}^*$ which implies that $(\tilde{x}^*,\tilde{\theta}^*)$ is not a fuzzy optimal solution of \eqref{fuzzy1:0}, which contradicts the hyphotesis.

To finish the proof, let us check that $\Theta(\{\tilde{c}_1^t\tilde{x}^*, \ldots, \tilde{c}_r^t\tilde{x}^*\})=\tilde{\theta}^*$. Since $\Theta(\{\tilde{c}_1^t\tilde{x}^*, \ldots, \tilde{c}_r^t\tilde{x}^*\})\lessapprox\tilde{\theta}$, let us assume that $\Theta(\{\tilde{c}_1^t\tilde{x}^*, \ldots, \tilde{c}_r^t\tilde{x}^*\})\neq\tilde{\theta}$, that is $\Theta(\{\tilde{c}_1^t\tilde{x}^*, \ldots, \tilde{c}_r^t\tilde{x}^*\})\preceq\tilde{\theta}$. Define $\tilde{\theta}^{**}=\Theta(\{\tilde{c}_1^t\tilde{x}^*, \ldots, \tilde{c}_r^t\tilde{x}^*\})$. Then, $(\tilde{x}^*,\tilde{\theta}^{**})$ is a feasible solution of \eqref{fuzzy1:0} which dominates $(\tilde{x}^*,\tilde{\theta}^*)$, contradicting its optimality.
\end{proof}

\section{A multiobjective reformulation for FFLPP}
\label{sec:3}
In the previous section we state the equivalence between two formulations of fully fuzzy minimax mixed integer programming problem, \eqref{fuzzy:0} and \eqref{fuzzy1:0}. The main advantage of \eqref{fuzzy1:0} is that it can be equivalently rewritten as a three-objective crisp mixed integer programming problem using the arithmetic of triangular fuzzy numbers. In what follows, we describe such a crisp formulation and derive some properties about it.

First let us recall some basic definitions on multiobjective optimization. Let us consider the following vector optimization problem:
\begin{align*}
v-\min & f(x)=(f_1(x), \ldots, f_r(x))\\
\mbox{s.t. } & x \in X \subset \R^n.
\end{align*}
where $f_j: X \rightarrow \R$ are certain functions and where $v-\min$ (from \textit{vector minimization}) means that the goal is to find the minimal elements with respect to the componentwise order in $\R^r$. A feasible solution $x^* \in X$ is said to be an efficient solution of the above problem if there not exists $x \in X$ such that $f_j(x^*) \leq f_j(x)$ for $j=1, \ldots, r$ with $f_j(x^*) < f_j(x)$ for some $j\in \{1, \ldots, r\}$.

\begin{theorem}\label{th:mo}
Let $(\tilde{x}\equiv (x^-,\hat x, x^+),\tilde{\theta}\equiv (\theta^-,\hat \theta, \theta^+))$ a fuzzy optimal solution of \eqref{fuzzy1:0}. Then, $\Big((x^-,\hat x, x^+),  (\theta^-,\hat \theta, \theta^+)\Big) \in \R^{n\times 3} \times \R^3$ is an efficient solution of the following three-objective mixed integer programming problem:
\begin{align}
v-\min& \;(\theta^-,\hat \theta, \theta^+)\label{mo}\tag{VMMLP}\\
\mbox{s.t. }& \dsum_{i=1}^n \left(\tilde{c_{ji}}\tilde{x_i}\right)^- \le {\theta}^-, \forall j=1, \ldots, r,\label{mo1}\\
& \dsum_{i=1}^n \widehat{\tilde{c_{ji}}\tilde{x_i}}\le \hat{\theta}, \forall j=1, \ldots, r,\label{mo2}\\
& \dsum_{i=1}^n \left(\tilde{c_{ji}}\tilde{x_i}\right)^+ \le {\theta}^+, \forall j=1, \ldots, r,\label{mo3}\\
& \dsum_{i=1}^n \left(\tilde{a_{ji}}\tilde{x_i}\right)^-  {\preceq\choose =}  \;\;\widehat{b}_j, \forall j=1, \ldots, m,\label{mo4}\\
& \dsum_{i=1}^n \widehat{\tilde{a_{ji}}\tilde{x_i}} {\preceq\choose =}  \;\;b_j^-, \forall j=1, \ldots, m,\label{mo5}\\
& \dsum_{i=1}^n \left(\tilde{a_{ji}}\tilde{x_i}\right)^+  {\preceq\choose =}  \;\;b_j^+, \forall j=1, \ldots, m,\label{mo6}\\
& {\theta}^- -\hat{{\theta}} \leq 0,\label{mo7}\\
& \hat{{\theta}}-{\theta}^+ \leq 0, \label{mo8}
\end{align}
\begin{align}
& \left(\tilde{c_{ji}}\tilde{x_i}\right)^--\widehat{\tilde{c_{ji}}\tilde{x_i}} \leq 0, \forall i=1, \ldots, n, j=1, \ldots, r ,\label{mo9}\\
& \widehat{\tilde{c_{ji}}\tilde{x_i}}-\left(\tilde{c_{ji}}\tilde{x_i}\right)^+ \leq 0, \forall i=1, \ldots, n, j=1, \ldots, r ,\label{mo10}\\
& x^-_i-\hat{x}_i \leq 0,\forall i \in N,\label{mo11}\\
& \hat{x}_i-x^+_i \leq 0,\forall i \in N, \label{mo12}\\
& x^-_i \geq 0,  \forall i \in N,\label{mo13}\\
& x_i^-=\hat x_i=x_i^+ \in\{0,1\}, \forall i \in Z.\label{mo14}
\end{align}
Furthermore, any efficient solution of \eqref{mo} induces a fuzzy optimal solution of \eqref{fuzzy1:0}. 
\end{theorem}

\begin{proof}
The proof follows by applying the arithmetic of triangular fuzzy numbers and the relations stated in Theorem \ref{order characterization}. Constraints \eqref{mo1}--\eqref{mo3} ensure the correct representation of the minimal upper bound fuzzy number. Constraints \eqref{mo4}--\eqref{mo6} are the original constraints of our fuzzy mixed integer programming problem when applied to the extremes of the triangular fuzzy numbers. With constraints \eqref{mo7} and \eqref{mo8}, \eqref{mo9} and \eqref{mo10}, and \eqref{mo11} and \eqref{mo12}, we assure the adequate representation of the triangular fuzzy numbers $\tilde{\theta}$, $\widetilde{c_j^t x}$ for all $j=1, \ldots, r$, and $\tilde{x}$, respectively.
\end{proof}

Observe that the multiobjective problem above involves the three components of the triangular fuzzy number that result from multiplying two triangular fuzzy number and that by Lemma \ref{lemma:1} its crisp expression depends on the signs of the extremes of the fuzzy numbers in $\tilde{c}$ (recall that $\tilde{x}$ are assumed to be nonnegative fuzzy numbers).

There exist several methods to generate efficient solutions of (VMMLP) (see, for instance, \cite{arana10,arana17}). Most popular methods are based on scalarization. In this regard, Ehrgott \cite{ehrgott2006} discussed the requirements of scalarizations to be used to get efficient solutions in multiobjective integer programs with linear objectives. One apply the weighted sum method, the $\epsilon$-constrained method, Benson's method, the augmented weighted Chebyshev mehtod or the weighted max-ordering method. All of them are either not able to generate all efficient solutions or might be extremely hard. In fact, Ehrgott \cite{ehrgott2006} pointed out that they all are $\mathcal{NP}$-hard, in general. To find all efficient solutions, he proposed the method of elastic constraints as a combination of good features of the weighted sum method and the $\epsilon$-constrained method. To a similar matter, recently, Kirlik and Sayin \cite{kirlik2014} has used two-stage formulations to avoid weakly efficient solutions as opposed to lexicographic optimization \cite{ben1980} employed by Laumanns et al. \cite{laumanns2006}. Kirlik and Sayin \cite{kirlik2014} has proposed a two-stage $\epsilon$-constrained method, in such a way that all efficient solutions can be found for a multiobjective integer program.

\section{Fuzzy Capacitated Center Facility Location Problem}
\label{sec:4}
In this section we apply the above framework to a classical problem in Location Analysis: the capacitated center facility location problem (CCFLP). We are given a set of customers $I=\{1, \ldots, n\}$ and a set of potential facilities $J=\{1, \ldots, m\}$. We are also given a set of allocation costs. For $i\in I$ and $j\in J$, $c_{ij}$ is the cost of allocating $i$ to $j$. Each customer $i$ is assumed to have a demand $d_i$ while each potential facility $j\in J$ is assumed to have a limited capacity $u_j$. A set-up cost is also considered for each facility $j\in J$, $f_j$.The goal is to find the set of potential facilities that must be open by minimizing the maximum allocation cost as well as the overall sum of set-up costs for the open facilities. The CCFLP  is particularly important in emergency situation because a risk-averse decision on the response time to customers is desired. The center (or minimax) facility location problems have been analyzed by different authors, using different techniques and have been applied to different problem (see \cite{minimaxloc1}). However, facility location problems has only been partially studied under a fuzzy environment (see \cite{fuzzyLoc}).

We use the decision variables $y_j = \left\{\begin{array}{cl}
1 & \mbox{if facility $j$ is open,}\\
0 & \mbox{otherwise}\end{array}\right.$ and $x_{ij}$ as the amount of demand of customer $i$ served from facility $j$, for $i\in I$ and $j\in J$. The problem is usually formulated as follows:
\begin{align}
\min & \max_{i=1, \ldots, n} \left\{\sum_{j \in J} c_{ij}x_{ij} \right\} + \dsum_{j\in J} f_jy_j\label{pc:0}\\
\mbox{s.t. } & \sum_{j\in J} x_{ij}=d_i, \forall i \in I,\label{pc:1}\\
& \sum_{i \in I}  x_{ij} \leq u_jy_j, \forall j \in J,\label{pc:2}\\
& x_{ij} \geq 0, \forall i \in I, j \in J,\nonumber\\
& y_j \in \{0,1\}, \forall j \in J.\nonumber
\end{align}
Observe that \eqref{pc:0} is the minimax objective function which tries to find the minimum of ther maximum allocation costs between customers and open facilities. Constraints \eqref{pc:1} assure that the whole demand of each customer is filled. Constraints \eqref{pc:2} ensures that in case the $j$-th facility of open, the capacity of such a facility is not overloaded.

Let us now consider the fuzzy version of the above problem. Let $\tilde{c} \in \mathcal{T}^{n\times m}_{\geq 0}$, $\tilde{d} \in \mathcal{T}^n_{\geq 0}$, $\tilde{u} \in \mathcal{T}^m_{\geq 0}$ and $\tilde{f} \in \mathcal{T}^m_{\geq 0}$ the matrix of allocation cost triangular fuzzy numbers, the fuzzy demands, the fuzzy capacities and the fuzzy set-up costs. Let us also consider that the fraction of demand of each customer served by each facility is a fuzzy number. To be realistic, constraint \eqref{pc:1} is consider as a crisp constraints assuring that the overall demand of each customer is satisfied. With this notation, the fully fuzzy version of the capacitated center location problem is:
\begin{align*}
\min & \;\;\Theta\left(\sum_{j \in J} \tilde{c}_{1j} \tilde{x}_{1j}, \ldots, \sum_{j \in J} \tilde{c}_{nj} \tilde{x}_{nj}  \right) + \dsum_{j\in J} \tilde{f}_jy_j\\
\mbox{s.t. } & \sum_{j\in J} \tilde{x}_{ij}=\tilde{d}_i, \forall i \in I,\\
& \sum_{i \in I}  \tilde{x}_{ij} \lessapprox \tilde{u}_jy_j, \forall j \in J,\\
& \tilde{x}_{ij} \gtrapprox 0, \forall i \in I, j \in J,\\
& y_j \in \{0,1\}, \forall j \in J.
\end{align*}
Note that the above problem although not written in the form of the fuzzy minimax problem \eqref{mm0}--\eqref{mmf} because its objective function. However, it can be adequately rewritten in the standard form since:
$$
\Theta\left(\sum_{j \in J} \tilde{c}_{1j} \tilde{x}_{1j}, \ldots, \sum_{j \in J} \tilde{c}_{nj} \tilde{x}_{nj}  \right) + \dsum_{j\in J} \tilde{f}_jy_j = \Theta\left(\sum_{j \in J} \tilde{c}_{1j} \tilde{x}_{1j}+ \dsum_{j\in J} \tilde{f}_jy_j, \ldots, \sum_{j \in J} \tilde{c}_{nj} \tilde{x}_{nj}+ \dsum_{j\in J} \tilde{f}_jy_j  \right)
$$

By Theorem \ref{th:mo}, the above problem can be equivalently rewritten as the following  three-objective mixed integer programming problem:
\begin{align*}
v-\min & \;\; (\theta^-, \hat \theta , \theta^+)\\
\mbox{s.t. } & \sum_{j \in J} {c}_{ij}^- {x}_{ij}^- + \dsum_{j\in J} f_j^-y_j\leq \theta^-, \forall i\in I,\\
& \sum_{j \in J} \hat{c}_{ij}\hat{x}_{ij} + \dsum_{j\in J}\hat f_j y_j\leq \hat{\theta}, \forall i\in I,\\
& \sum_{j \in J} {c}_{ij}^+ {x}_{ij}^+ +\dsum_{j\in J} f_j^+y_j \leq \theta^+, \forall i\in I,\\
&\sum_{j\in J} {x}_{ij}^-=d_i^-, \forall i \in I,\\
&\sum_{j\in J} \hat{x}_{ij}=\hat d_i, \forall i \in I,\\
&\sum_{j\in J} {x}_{ij}^+=d_i^+, \forall i \in I,\\
& \sum_{i \in I} {d}_i^- {x}_{ij}^- \leq {u}_j^- y_j, \forall j \in J,\\
& \sum_{i \in I} \hat{d}_i\hat {x}_{ij} \leq \hat{u}_j y_j, \forall j \in J,\\
& \sum_{i \in I} {d}_i^+ {x}_{ij}^+ \leq {u}_j^+ y_j, \forall j \in J,
\end{align*}
\begin{align*}
& \theta^- \leq \hat{\theta} \leq \theta^+,\\
& c_{ij}^- x_{ij}^- - \hat{c}_{ij}\hat{x}_{ij} \leq 0, \forall i \in I, j\in J,\\
& \hat{c}_{ij}\hat{x}_{ij} - c_{ij}^+ x_{ij}^+ \leq 0, \forall i \in I, j\in J,\\
& x_{ij}^- - \hat{x}_i \leq 0, \forall i \in I,  j \in J,\\
& \hat{x}_{ij} - x_i^+ \leq 0, \forall i \in I, j \in J,\\
& {x}_{ij}^- \geq 0, \forall i \in I, j \in J,\\
& y_j \in \{0,1\}, \forall j \in J.
\end{align*}

Observe that in this case all the triangular fuzzy numbers are assumed to be nonnegative, and then the multiplication rule is easily derived.

\begin{example}\label{ex:f}
Let us consider the randomly generated fuzzy triangular fuzzy numbers given in Table \ref{param} for the demands, capacities and set-up costs of the $6$ customers (which also act as potential facilities) drawn in Figure \ref{ex:fig}. The transportation costs are shown in Table \ref{costs}.

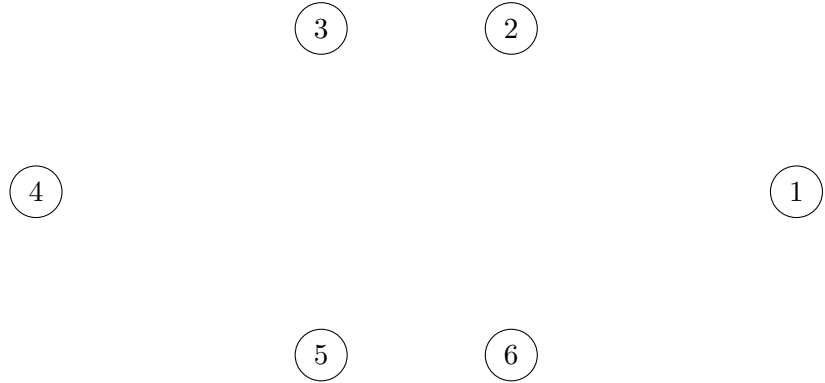
\begin{figure}[h]
\begin{center}
\begin{tikzpicture}[scale=0.05]

\coordinate(X1) at (100,   0);
\coordinate(X2) at (25,   43.3012702);
\coordinate(X3) at (-25,   43.3012702);
\coordinate(X4) at (-100,   0);
\coordinate(X5) at ( -25,  -43.3012702);
\coordinate(X6) at (25,  -43.3012702);

\node[circle,draw](x1) at (X1) {$1$};
\node[circle,draw](x2) at  (X2){$2$};
\node[circle,draw](x3) at  (X3){$3$};
\node[circle,draw](x4) at  (X4){$4$};
\node[circle,draw](x5) at  (X5){$5$};
\node[circle,draw](x6) at  (X6){$6$};
\end{tikzpicture}
\end{center}
\caption{Demand and Potential facility points of Example \ref{ex:f}.\label{ex:fig}}
\end{figure}

\begin{table}[h]
\begin{center}
\begin{tabular}{cc}\\\hline
Parameters & Values\\\hline
$d^-$ & $[1.11,0.28,42.15,18.63,3.06,12.30]$\\
       $\hat d\;\;$ & $[23.00,33.00,46.00,20.00,20.00,37.00]$\\
       $d^+$ & $[23.53,37.32,50.90,28.63,22.15,42.10]$\\\hline
$u^-$ & $[8.78,5.72,53.69,90.14,5.75,44.22]$\\
$\hat u\;\;$ & $[66.00,41.50,74.25,113.25,34.25,90.50]$\\
$u^+$ & $[69.66,45.43,98.63,146.76,49.13,122.92]$\\\hline
$f^-$ & $[358.73,271.28,288.22,448.31,491.00,0.44]$\\
$\hat f^-\;\;$ & $[433.00,524.00,561.00,691.00,520.00,487.00]$\\
$f^+$ & $[ 437.51,1019.19,1003.98,893.36,769.90,535.55]$\\\hline
       \end{tabular}
       \caption{Fuzzy demands, capacities and set-up costs of Example \ref{ex:f}.\label{param}}
              \end{center}
       \end{table}
\renewcommand{\tabcolsep}{0.05cm}

\begin{table}[h]
\begin{center}
{\scriptsize
\begin{tabular}{|cccccc|cccccc|cccccc|}\hline
\multicolumn{6}{|c}{$c^-$} & \multicolumn{6}{|c}{$\hat c$} & \multicolumn{6}{|c|}{$c^+$} \\\hline
0.00 &27.41 &41.62 &73.96 &42.16 &34.28 &0.00 &43.30 &66.14 &100.00 &66.14 &43.30 &0.00 &44.67 &77.96 &103.33 &92.50 &44.58 \\
39.58 &0.00 &20.23 &40.64 &34.67 &39.79 &43.30 &0.00 &25.00 &66.14 &50.00 &43.30 &49.30 &0.00 &27.50 &81.65 &50.25 &53.08 \\
62.95 &21.40 &0.00 &32.50 &29.51 &47.40 &66.14 &25.00 &0.00 &43.30 &43.30 &50.00 &79.96 &34.74 &0.00 &44.81 &55.32 &55.99 \\
97.10 &52.82 &39.56 &0.00 &27.13 &41.70 &100.00 &66.14 &43.30 &0.00 &43.30 &66.14 &107.10 &79.37 &45.14 &0.00 &56.90 &72.44 \\
59.61 &40.01 &36.93 &26.23 &0.00 &21.90 &66.14 &50.00 &43.30 &43.30 &0.00 &25.00 &82.13 &53.13 &53.93 &56.68 &0.00 &33.60 \\
35.92 &39.17 &43.94 &49.19 &21.47 &0.00 &43.30 &43.30 &50.00 &66.14 &25.00 &0.00 &58.58 &48.03 &53.65 &70.17 &33.47 &0.00 \\\hline
\end{tabular}}
       \caption{Allocation costs of Example \ref{ex:f}.\label{costs}}
              \end{center}
       \end{table}

The crisp solution of the minmax capacitated location problem using the center values of the demands ($\hat d$), capacities ($\hat u$), set-up costs ($\hat f$) and allocation costs ($\hat c$), we get the solution drawn in Figure \ref{ex:crisp} where the black filled nodes indicate open facilities, lines connecting customers and facilities when a nonzero demand is served to the customer by the facility and the number over the line is such a demand.

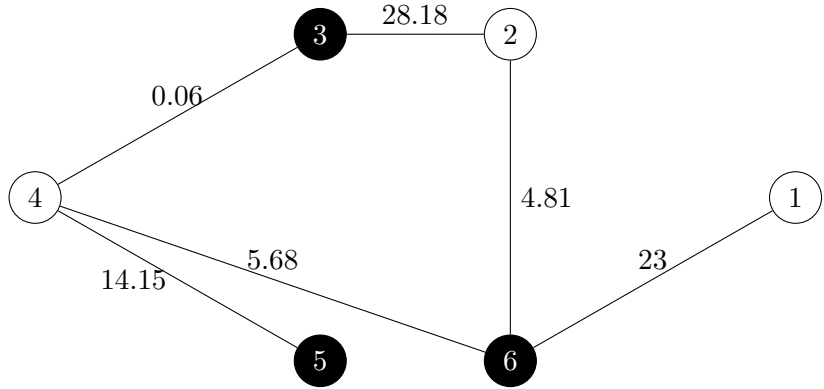
\begin{figure}[H]
\begin{center}
\begin{tikzpicture}[scale=0.05]

\coordinate(X1) at (100,   0);
\coordinate(X2) at (25,   43.3012702);
\coordinate(X3) at (-25,   43.3012702);
\coordinate(X4) at (-100,   0);
\coordinate(X5) at ( -25,  -43.3012702);
\coordinate(X6) at (25,  -43.3012702);

\node[circle,draw](x1) at (X1) {$1$};
\node[circle,draw](x2) at  (X2){{$2$}};
\node[circle,fill,draw](x3) at  (X3){\color{white}$3$};
\node[circle,draw](x4) at  (X4){$4$};
\node[circle,fill,draw](x5) at  (X5){\color{white}$5$};
\node[circle,fill,draw](x6) at  (X6){\color{white}{$6$}};

\draw (x1)-- node [above]{$23$} (x6);
\draw (x2)-- node [above]{$28.18$} (x3);
\draw (x2)-- node [right]{$4.81$} (x6);
\draw (x4)-- node [above]{$0.06$} (x3);
\draw (x4)-- node [left]{$14.15$} (x5);
\draw (x4)-- node [above]{$5.68$} (x6);
\end{tikzpicture}
\end{center}
\caption{Crisp solution for Example \ref{ex:f}.\label{ex:crisp}}
\end{figure}

Using Gurobi, we generate some fuzzy solutions for the fuzzy version of the problem, by using hierarchical methods (see \cite{gurobi}). Some of them are drawn in Figure \ref{fig:fuzzy}.

\begin{center}
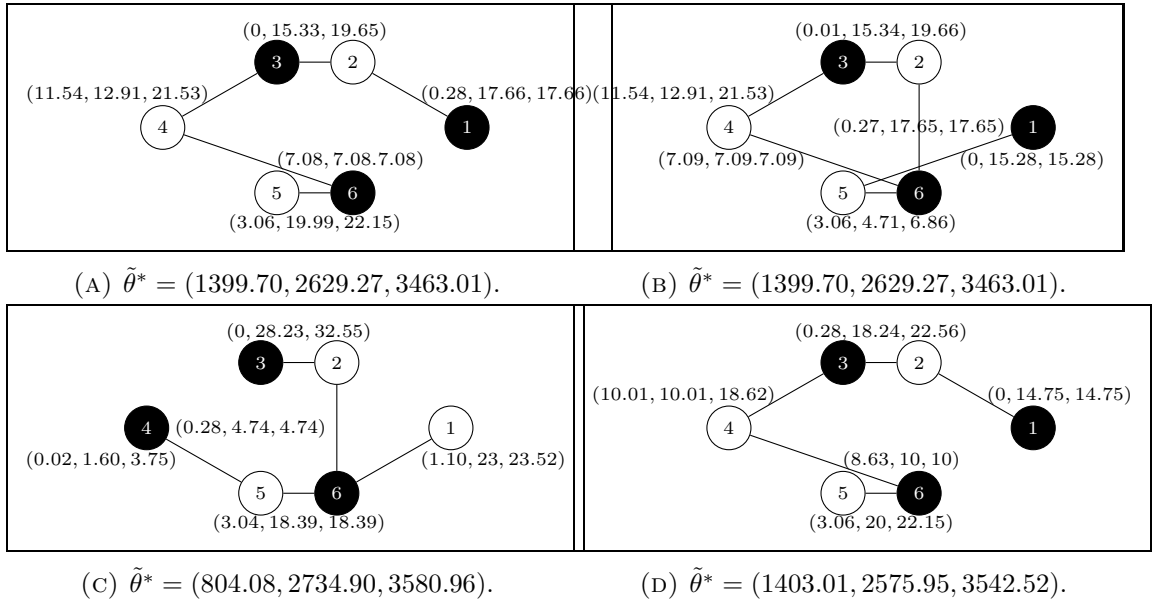
\begin{figure}[H]
\begin{subfigure}[b]{0.5\textwidth}
\fbox{\begin{tikzpicture}[scale=0.02]

\coordinate(X1) at (100,   0);
\coordinate(X2) at (25,   43.3012702);
\coordinate(X3) at (-25,   43.3012702);
\coordinate(X4) at (-100,   0);
\coordinate(X5) at ( -25,  -43.3012702);
\coordinate(X6) at (25,  -43.3012702);

\node[circle,fill,draw](x1) at (X1) {\tiny \color{white}$1$};
\node[circle,draw](x2) at  (X2){\tiny{$2$}};
\node[circle,fill,draw](x3) at  (X3){\tiny\color{white}$3$};
\node[circle,draw](x4) at  (X4){\tiny{$4$}};
\node[circle,draw](x5) at  (X5){\tiny$5$};
\node[circle,fill,draw](x6) at  (X6){\tiny\color{white}{$6$}};

\draw (x2)-- node [right]{\tiny $(0.28,17.66,17.66)$} (x1);
\draw (x2)-- node [above=4pt]{\tiny $(0,15.33,19.65)$} (x3);
\draw (x4)-- node [left]{\tiny $(11.54,12.91,21.53)$} (x3);
\draw (x4)-- node [right=3pt]{\tiny $(7.08,7.08.7.08)$} (x6);
\draw (x5)-- node [below=4pt]{\tiny $(3.06,19.99,22.15)$} (x6);
\end{tikzpicture}}
\caption{$\tilde{\theta}^*=(1399.70, 2629.27, 3463.01)$.\label{a}}
\end{subfigure}~\begin{subfigure}[b]{0.5\textwidth}
\fbox{\begin{tikzpicture}[scale=0.02]

\coordinate(X1) at (100,   0);
\coordinate(X2) at (25,   43.3012702);
\coordinate(X3) at (-25,   43.3012702);
\coordinate(X4) at (-100,   0);
\coordinate(X5) at ( -25,  -43.3012702);
\coordinate(X6) at (25,  -43.3012702);

\node[circle,fill,draw](x1) at (X1) {\tiny \color{white}$1$};
\node[circle,draw](x2) at  (X2){\tiny{$2$}};
\node[circle,fill,draw](x3) at  (X3){\tiny\color{white}$3$};
\node[circle,draw](x4) at  (X4){\tiny{$4$}};
\node[circle,draw](x5) at  (X5){\tiny$5$};
\node[circle,fill,draw](x6) at  (X6){\tiny\color{white}{$6$}};

\draw (x2)-- node [above=4pt]{\tiny $(0.01,15.34,19.66)$} (x3);
\draw (x2)-- node  {\tiny $(0.27,17.65,17.65)$} (x6);
\draw (x4)-- node [left]{\tiny $(11.54,12.91,21.53)$} (x3);
\draw (x4)-- node [left=3pt]{\tiny $(7.09,7.09.7.09)$} (x6);
\draw (x5)-- node [below=4pt]{\tiny $(3.06,4.71,6.86)$} (x6);
\draw (x5)-- node [right=4pt]{\tiny $(0,15.28,15.28)$} (x1);
\end{tikzpicture}}
\caption{$\tilde{\theta}^*=(1399.70, 2629.27, 3463.01)$.\label{b}}
\end{subfigure}

\begin{subfigure}[b]{0.5\textwidth}
\fbox{\begin{tikzpicture}[scale=0.02]

\coordinate(X1) at (100,   0);
\coordinate(X2) at (25,   43.3012702);
\coordinate(X3) at (-25,   43.3012702);
\coordinate(X4) at (-100,   0);
\coordinate(X5) at ( -25,  -43.3012702);
\coordinate(X6) at (25,  -43.3012702);

\node[circle,fill,draw](x4) at (X4) {\tiny \color{white}$4$};
\node[circle,draw](x2) at  (X2){\tiny{$2$}};
\node[circle,fill,draw](x3) at  (X3){\tiny\color{white}$3$};
\node[circle,draw](x1) at  (X1){\tiny{$1$}};
\node[circle,draw](x5) at  (X5){\tiny$5$};
\node[circle,fill,draw](x6) at  (X6){\tiny\color{white}{$6$}};

\draw (x2)-- node [above=4pt]{\tiny $(0,28.23,32.55)$} (x3);
\draw (x2)-- node [left]{\tiny $(0.28,4.74,4.74)$} (x6);
\draw (x5)-- node [left=6pt]{\tiny $(0.02,1.60,3.75)$} (x4);
\draw (x5)-- node [below=4pt]{\tiny $(3.04,18.39,18.39)$} (x6);
\draw (x1)-- node [right=6pt]{\tiny $(1.10,23,23.52)$} (x6);
\end{tikzpicture}}
\caption{$\tilde{\theta}^*=(804.08, 2734.90, 3580.96)$.\label{c}}
\end{subfigure}~\begin{subfigure}[b]{0.5\textwidth}
\fbox{\begin{tikzpicture}[scale=0.02]

\coordinate(X1) at (100,   0);
\coordinate(X2) at (25,   43.3012702);
\coordinate(X3) at (-25,   43.3012702);
\coordinate(X4) at (-100,   0);
\coordinate(X5) at ( -25,  -43.3012702);
\coordinate(X6) at (25,  -43.3012702);

\node[circle,fill,draw](x1) at (X1) {\tiny \color{white}$1$};
\node[circle,draw](x2) at  (X2){\tiny{$2$}};
\node[circle,fill,draw](x3) at  (X3){\tiny\color{white}$3$};
\node[circle,draw](x4) at  (X4){\tiny{$4$}};
\node[circle,draw](x5) at  (X5){\tiny$5$};
\node[circle,fill,draw](x6) at  (X6){\tiny\color{white}{$6$}};

\draw (x2)-- node [right]{\tiny $(0,14.75,14.75)$} (x1);
\draw (x2)-- node [above=4pt]{\tiny $(0.28,18.24,22.56)$} (x3);
\draw (x4)-- node [left]{\tiny $(10.01,10.01,18.62)$} (x3);
\draw (x4)-- node [right=3pt]{\tiny $(8.63,10,10)$} (x6);
\draw (x5)-- node [below=4pt]{\tiny $(3.06,20,22.15)$} (x6);
\end{tikzpicture}}
\caption{$\tilde{\theta}^*=(1403.01, 2575.95, 3542.52)$.\label{d}}
\end{subfigure}
\caption{Some fuzzy solutions of Example \ref{ex:f}.\label{fig:fuzzy}}
\end{figure}
\end{center}

The first observation after solving the crisp and the fuzzy versions of the problems is that  the shapes of the obtained networks when some kind of imprecision is assumed may be different from the crisp solutions.  Actually, the amounts of demands served for each customer by each open facility in the fuzzy case are, in most of the cases,  not crisp amounts but pure triangular fuzzy numbers. Also, observe that the solutions drawn in figures \ref{fig:fuzzy}.\eqref{sub@a} and \ref{fig:fuzzy}.\eqref{sub@b}  have the same values for the $\tilde{\theta}^*$ because of the minimax objectives which only account for the maximal cost, which serves as an upper bound for the rest of the costs.
\end{example}

\section{Conclusions}
\label{sec:5}

In this paper we analyze a fully fuzzy version of minimax mixed integer linear programming problems. We extend the minimax objective my means of defining the maximal upper bound of a set of fuzzy numbers. With such a notion, we formulate minimax fuzzy mixed integer programming problems in which all the parameters and continuous variables are triangular fuzzy numbers. We provide two equivalent fuzzy formulations for the problem and state that solving them is also equivalent of solving a three-objective mixed integer programming problem. Finally, we apply the proposed approach to a well-known problem in Locational Analysis: the  capacitated center facility location problem.

\section*{Acknowledgments}
The second author was partially supported by the research projects MTM2016-74983-C2-1-R (MINECO, Spain) and PP2016-PIP06 (Universidad
de Granada) and the research group SEJ-534 (Junta de Andaluc\'ia). 


\end{document}